\documentclass[11pt,letterpaper]{amsart}
\usepackage[utf8]{inputenc}
\usepackage[margin=1in]{geometry}
\usepackage[dvipsnames,svgnames]{xcolor}
\usepackage{todonotes}
\usepackage[foot]{amsaddr}
\usepackage{graphicx}
\usepackage{xspace}
\usepackage{amsmath}
\usepackage{amsthm,amssymb}
\usepackage{paralist}
\usepackage{mathrsfs}
\usepackage{url}
\usepackage{svg}
\usepackage{hyperref}

\newif\ifdraft\drafttrue

\ifdraft

\newcommand\js[1]{\todo[inline,size=\scriptsize,backgroundcolor=SpringGreen]{#1 - \textbf{Joseph}}}
\newcommand\gl[1]{\todo[inline,size=\scriptsize,backgroundcolor=Pink]{#1 - \textbf{Guillaume}}}

\else

\newcommand\js[1]{}
\newcommand\gl[1]{}

\fi

\makeatletter
\renewcommand{\email}[2][]{%
  \ifx\emails\@empty\relax\else{\g@addto@macro\emails{,\space}}\fi%
  \@ifnotempty{#1}{\g@addto@macro\emails{\textrm{(#1)}\space}}%
  \g@addto@macro\emails{#2}%
}
\makeatother

\newcommand\ie{{i.e.,}\xspace}

\newtheorem{theorem}{Theorem}
\newtheorem*{theorem*}{Theorem}
\newtheorem{definition}{Definition}

\newtheorem{lemma}{Lemma}
\newtheorem{claim}{Claim}

\newtheorem{remark}{Remark}

\newcommand\textbfit[1]{{\bf\em #1}}
\newcommand{\defin}[1]{\textbfit{\boldmath #1}}

\newcommand{\N}{\mathbb{N}}

\newcommand{\F}{\mathcal{F}} 
\newcommand{\G}{\mathcal{G}} 
\newcommand{\dpowparam}{4} 
\newcommand{\semifinaldpow}{7} 
\newcommand{\finaldpow}{9} 
\newcommand{\maxlemconst}{13} 

\newcommand{\comp}[1]{\overline{#1}} 
\newcommand{\mg}[1]{m(#1)}

\begin{document}

\title{$d$-Galvin families}

\author{Johan Håstad}
\address[Johan Håstad]{KTH Royal Institute of Technology, Stockholm, Sweden}
\email{johanh@kth.se}

\author{Guillaume Lagarde}
\address[Guillaume Lagarde]{KTH Royal Institute of Technology, Stockholm, Sweden}
\email{glagarde@kth.se}

\author{Joseph Swernofsky}
\address[Joseph Swernofsky]{KTH Royal Institute of Technology, Stockholm, Sweden}
\email{josephsw@kth.se}

\date{\today}

\begin{abstract}

The Galvin problem asks for the minimum size of
  a family $\F \subseteq \binom {[n]} {n/2}$ with the property that,
  for any set $A$ of size $\frac n 2$, there is a set $S \in \F$ which
  is balanced on $A$, meaning that $|S \cap A| = |S \cap \comp A|$. We
  consider a generalization of this question that comes from a possible approach in
  complexity theory. In the generalization the required property is, for any $A$, to be able to find $d$ sets from a family
  $\F \subseteq \binom {[n]} {n/d}$ that form a partition of $[n]$
  and such that each part is balanced on $A$. We construct such
  families of size polynomial in the parameters $n$ and $d$.  
\end{abstract}

\maketitle

\smallskip
\section{Introduction}
\label{section:intro}

\subsection{Galvin problem}
The starting point of this paper is a question raised by Galvin in
extremal combinatorics. Given two sets $A$ and $S$, we say that $S$ is
\defin{balanced on $A$} if $|S \cap A| = \frac{|S|}{2}$.

\begin{figure}[h!]
  \label{fig:exGalvin}
    \centering
    \includegraphics[scale=0.6]{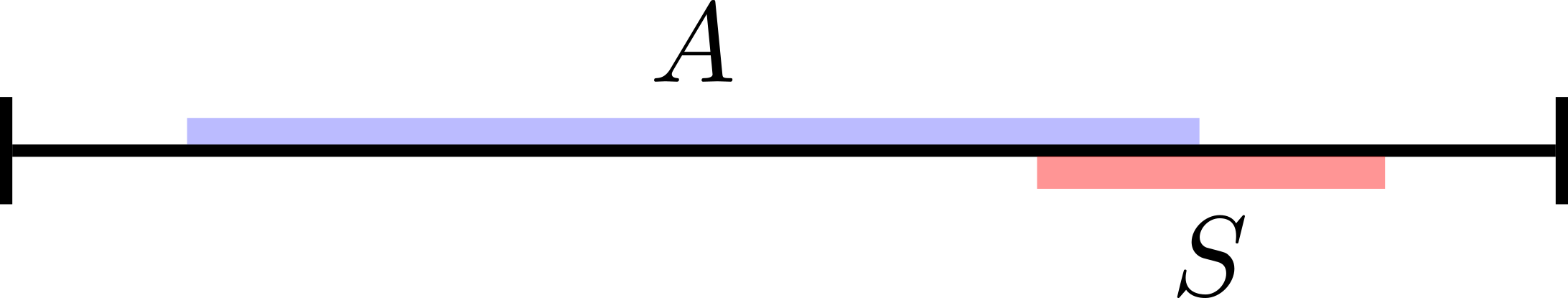}
    \caption{$S$ balanced on $A$}
  \end{figure}

\begin{definition}[Galvin family]
  If $4 \mid n$, a family $\F \subseteq \binom {[n]} {n/2}$ is said to
  be \defin{Galvin} if for any $A \in \binom {[n]} {n/2}$ there
  exists a set $S \in \F$ which is balanced on $A$ (\ie
  $|S \cap A| = \frac{n}{4}$).
\end{definition}
The \defin{Galvin problem} asks for the minimal size, denoted by
$\mg{n}$, of a Galvin family. An upper bound of
$\mg{n} \leq \frac{n}{2}$ follows from the family given by the sets
$S_i = \{i, i+1, \dots, i + \frac{n}{2} -1 \}$ for $i \in
[n/2]$. Lower bounds for the size of Galvin families are more subtle. An
easy counting argument shows that
$m(n) \geq \frac{\binom {n} {n/2}}{\binom {n/2} {n/4}^2} =
\Theta(\sqrt{n})$, which is far from $n/2$. Frankl and Rödl~\cite{FR}
established that $\mg{n} \geq \epsilon n$ for some $\epsilon > 0$
whenever $\frac{n}{4}$ is odd, as a corollary to a strong result in
extremal set theory. This linear bound was later strengthened by
Enomoto, Frankl, Ito and Nomura~\cite{EF} to $\mg{n} = n/2$, with the
same parity constraint, thus showing the optimality of the
construction in this special case. Later, using Gröbner basis methods
and linear algebra, Hegedűs~\cite{H09} obtained that
$\mg{n} \geq \frac{n}{4}$ whenever $\frac{n}{4} > 3$ is a prime.

\begin{figure}[h!]
  \label{fig:exGalvin}
    \centering
    \includegraphics[scale=0.23]{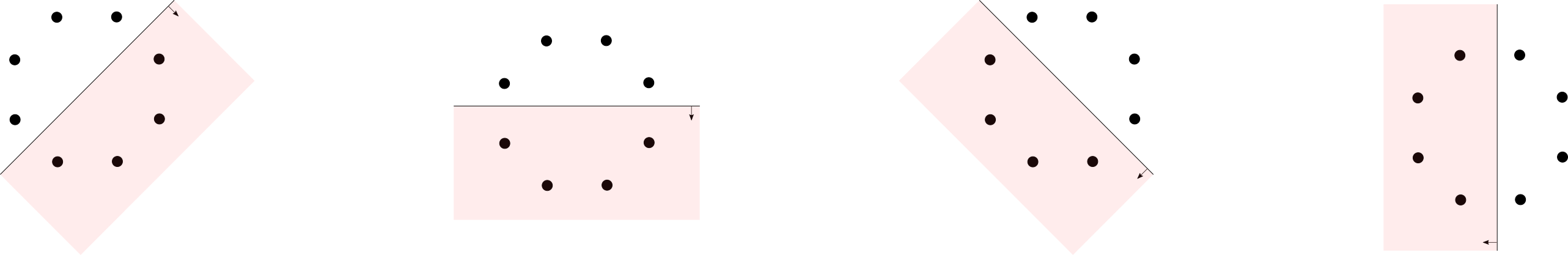}
    \caption{A Galvin family for $n=8$ consisting of 4 sets}
  \end{figure}

\subsection{Generalizations and related works}
Surprisingly, problems closely related to the one of Galvin proved
useful in arithmetic complexity theory, in order to give lower bounds
on the size of arithmetic circuits computing some target
polynomials. This connection was first noticed by Jansen~\cite{J08}, and was
recently successfully used in a paper by Alon et al.~\cite{AlonKV18}. There the elements of
the Galvin family $\F$ are allowed to be sets of size between $2 \tau$
and $n - 2\tau$ ($\tau$ being an integer). Furthermore, for a given
$A \in \binom {[n]} {n/2}$ instead of asking for the existence
of a set $S \in \F$ perfectly balanced on $A$ the authors look for a
set $S$ which is nearly balanced, \ie
$\left||S \cap A| - \frac{|S|}{2}\right| < \tau$ for the same $\tau$. For this setting,
Alon, Kumar and Volk~\cite{AlonKV18} showed, using the so-called
polynomial method, that $\mg{n} \geq \Omega(n/\tau)$.

Alon, Bergmann, Coppersmith, and Odlyzko~\cite{AlonBCO88} investigate
a problem dealing with $\{-1,+1\}$ vectors which looks similar to the
Galvin one. When rephrasing it as an extremal problem over sets, it
reads as follows: what is the minimal number $K(n,c)$ on the size of a
family $\F \subseteq \mathcal P([n])$ such that the following holds
$$\forall A \subseteq [n], \exists S \in \F,
\left | |\comp A \triangle S| - |A \triangle S| \right | \leq c,$$
where $\triangle$ denotes the symmetric difference. Setting $c = 0$
and asking all sets to be of size $n/2$ is exactly Galvin
problem. However, it does not seem to be any evident dependencies
between the two problems.

We consider here a different type of generalization.  Asking for a set
$S \in \F$ to be balanced on $A \in \binom {[n]} {n/2}$ is equivalent
(up to a factor $2$ in the family size) to ask for a partition of
$[n]$ in two parts, namely $(S, \comp S)$, such that each part is
balanced on $A$ and such that $S$, $\comp S$ are elements of
$\F$. Instead of splitting $[n]$ in two parts, we look for partitions
that involve more sets. Introducing a parameter $d \in \N$, we want,
for a given $A$, to be able to find $d$ sets in $\F$ that form a
partition of $[n]$ and such that each set is balanced on $A$.

The
original motivation for considering this generalization stems from
arithmetic circuits. There, an open question is to know whether there
is a separation between two models of computation called multilinear
algebraic branching programs (ml-ABPs) and multilinear circuits
(ml-circuits). By ``separation'', we mean that there is some specific
polynomial $f$ that can be computed by a small ml-circuit but any
ml-ABP for $f$ must be of size superpolynomial in the degree and the
number of variables of $f$. Proving that any generalized Galvin
families (\ie with $d$ parts in the partitions -- see below for a
formal definition) must be of superpolynomial size (in $n$ the size of
the ground set, and $d$ the number of parts) would imply a separation
between ml-ABPs and ml-circuits. Since our main result is to prove
that generalized Galvin families of polynomial size exist, this
approach is unfortunately not promising. Note that this does not call
into question either the plausible separation between ml-ABPs and
ml-circuits or the approach through a proof that ml-ABPs cannot
compute efficiently so-called ``full rank polynomials''. This only
rules out a specific approach to tackling the question of knowing
whether ml-ABPs can efficiently compute full rank
polynomials. However, we believe that the construction is of intrinsic
combinatorial interest.

\section{$d$-Galvin families}
\subsection{Definition}
We start with the formal definition of generalized Galvin families:

\begin{definition}[$d$-Galvin families]
  Given two integers $d,n \in \N$ such that $2d \mid n$, we say that a
  family $\F \subseteq \binom {[n]} {\frac{n}{d}}$ is
  \defin{$d$-Galvin} if for any $A \in \binom {[n]} {n/2}$, \defin{$A$
  is handled by $\F$}, meaning that there exist $d$ sets
  $S_1,\dots,S_d \in \F$ such that:
  \begin{itemize}
  \item The $S_i$ form a partition of $[n]$,
  \item Each $S_i$ is balanced on $A$ (\ie $|S_i \cap A| = \frac{n}{2d}$).
  \end{itemize}
\end{definition}

\begin{figure}[h!]
  \label{fig:exGalvin}
    \centering
    \includegraphics[scale=0.5]{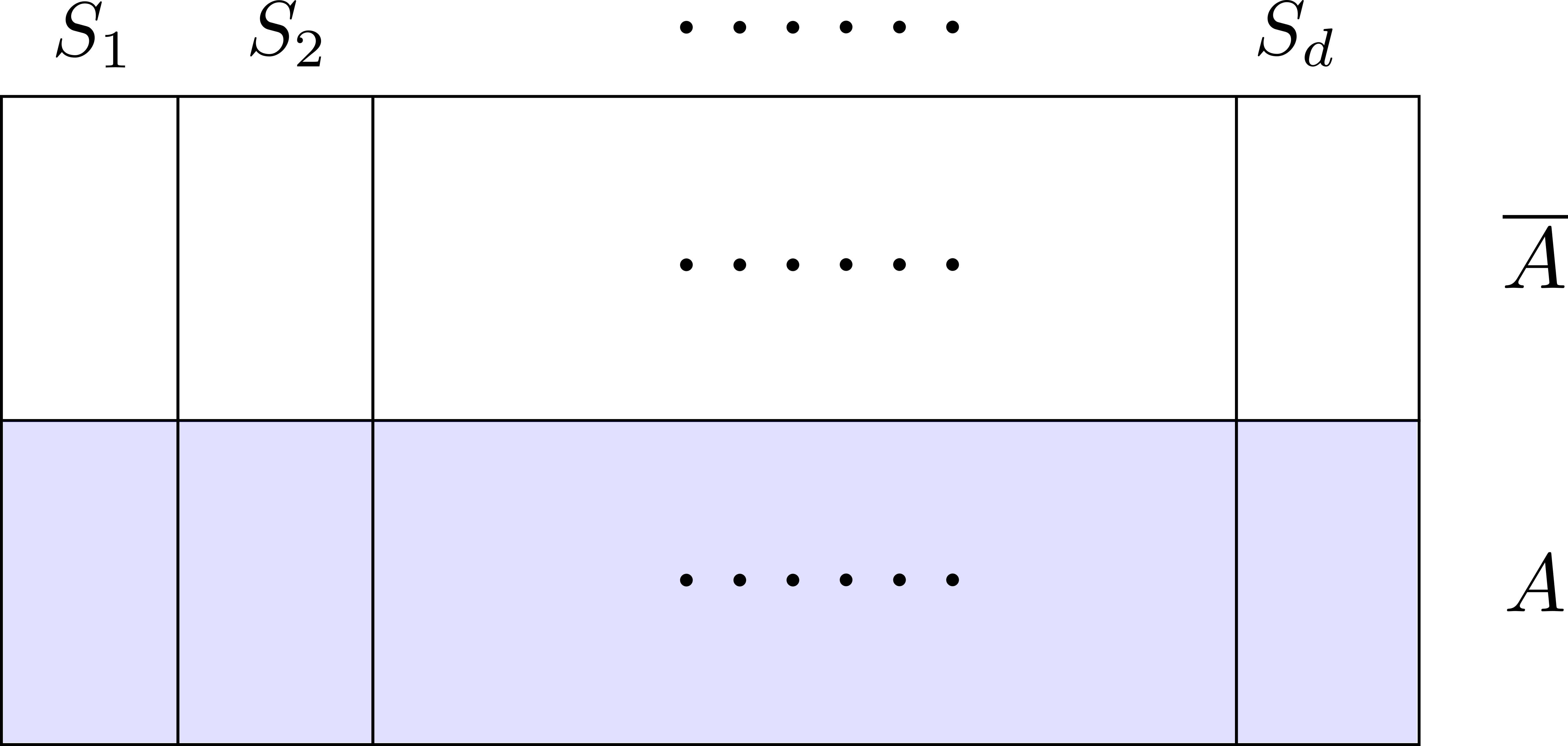}
    \caption{Set $A$ handled by a partition $S_1, S_2, \dots S_d$}
  \end{figure}

\begin{remark}
  Note that a $2$-Galvin family is simply a Galvin family
  (up to adding the complements of any set in the family).
\end{remark}

Somewhat surprisingly, small $d$-Galvin families exist.
  
\begin{theorem}
\label{thm:main}
For any $d, n \in \N$ such that $2d \mid n$, there exists a $d$-Galvin
family of size $\tilde\Theta(n^2 d^\finaldpow)$.
\end{theorem}

Here $\tilde\Theta(f(n,d))$ is some function $g$ such that
  $f(n,d) (\ln f(n,d))^{c_1} \leq g(n,d) \leq f(n,d) (\ln f(n,d))^{c_2}$ for
  some integers $c_1, c_2$. The next section is devoted to the
construction of a $d$-Galvin family, yielding a proof of the main
theorem.  

\subsection{Proof of Theorem~\ref{thm:main}}

For technical reasons, we need to distinguish two cases in the proof
of Theorem~\ref{thm:main}: we start by giving a construction when $d$
is reasonably small, then we show how to adapt it to handle larger $d$.

\noindent{\textbf{First case: }$d < \frac{n}{(\ln n)^3}$}

The overall idea is to construct a family $\F$ of size
$\tilde \Theta(n d^{\finaldpow})$ such that a random set
$A \in \binom {[n]} {n/2}$ is handled by $\F$ with probability at
least $1/2$. Taking the random family $\G$ which is the union of $n$
independent such $\F$ increases this probability to at least
$1- 2^{-n}$. By the union bound, the probability that $\G$ handles all
sets $A$ is non-zero, yielding the existence of the desired family. We
now focus on the construction of such a family $\F$.

\smallskip

\noindent \textbf{Construction of $\F$}

For a set $X$, we use the notation $A \sim X$ to denote that $A$ is a 
set chosen uniformly at random from $X$. We let $k := \frac n {2d}$ for the rest of the paper.
\begin{lemma}
\label{lem:randomf}
When $d < \frac{n}{(\ln n)^3}$, there is a family
$\F \subseteq \binom {[n]} {2k}$ of size
$\tilde \Theta(n d^\finaldpow)$ such that

 $$\Pr_{A \sim \binom {[n]} {n/2}}(A \text{ is handled by } \F) \geq 1/2$$

\end{lemma}

Before going into the construction, let us see how we can prove the
main theorem, with Lemma~\ref{lem:randomf} in hand.
\begin{proof}[Proof of Theorem~\ref{thm:main}, first case]
  Let $\sigma_1, \dots, \sigma_n$ be $n$ permutations of $[n]$, chosen
  uniformly at random. For any of these, construct the family
  $\F_{\sigma_i} = \sigma_i(\F)$, \ie the family from
  Lemma~\ref{lem:randomf} where any element $e \in [n]$ has been
  replaced by $\sigma_i(e)$. Consider the family
  $\G := \cup_{i \in [n]} \F_{\sigma_i}$. We aim to prove that $\G$ is
  $d$-Galvin with non-zero probability. Given a set $A$, let
  $\mathrm{H_i}$ be the event: ``$A$ is handled by
  $\F_{\sigma_i}$''. $\mathrm{H_i}$ is equivalent to
  ``$\sigma_i^{-1}(A)$ is handled by $\F$''. As $\sigma_i^{-1}(A)$ is a
  uniformly random set independent from $\sigma_{i'}^{-1}(A)$ for
  $i \neq i'$, this proves the independence between the events
  $\mathrm{H_i}$. From this we conclude
  $$\Pr_{A \sim \binom {[n]} {n/2}}(\forall i \in [n], A \text{ is not
    handled by } \F_{\sigma_i}) \le 2^{-n}$$ Thus, by the union bound,
  there is a non-zero probability that $\G$ handles all sets $A$,
  concluding the proof of the theorem.

\end{proof}

The rest of the section consists of a proof of
  Lemma~\ref{lem:randomf}. The overall strategy is to divide the
  elements of $[n]$ into buckets, denoted by $\chi_i$, and build the
  sets $S$ from any pair of buckets $(\chi_i,\chi_j)$. Suppose the
  amount by which these buckets are unbalanced on $A$ are $R_i$ and
  $R_j$ respectively. If half the elements of $S$ are chosen from
  bucket $\chi_i$ and half from bucket $\chi_j$ then the amount by
  which $S$ is unbalanced on $A$ will be close to a normal
  distribution with expectation depending on $R_i$ and $R_j$. By showing a
  good upper bound on the $R_i$, the probability that $S$ is balanced is
  reasonably large, and picking only polynomially many random sets $S$ is
  sufficient. In fact, we must be slightly more careful because the
  bucket errors accumulate as we pick many sets $S$. Fortunately, we
  can manage this by taking an ordering $\pi$ of the buckets such that
  the error of $\cup_{j \leq i} \chi_{\pi(j)}$ stays small for all $i$.

\begin{proof}[Proof of Lemma~\ref{lem:randomf}]
  First, we divide $[n]$ into several intervals (recall that $k = \frac n {2d}$).
\begin{itemize}
\item $\chi_0 = (0, k]$,
\item $\chi_i = ((2i-1)k, (2i+1) k]$
  for $i \in [d-1]$,
\item $\chi_d = ((2d-1) k, n]$.
\end{itemize}
For $i \in [d-1]$ we create sets $G_i = \{T_i^h, h \in [1,r] \}$ by sampling independently
  $r = \tilde \Theta(n^{1/2} d^{\semifinaldpow / 2})$ subsets $T_i^h \sim \binom {\chi_i}{k}$
and adding them to $G_i$. For technical reasons, we let $G_0$ to be
the singleton $\{\emptyset\}$ and $G_d = \{\chi_d\}$.  Finally let
$\F = \{(\comp {T_i^h} \cup T_{j}^{l}: i,j \in [0, d], T_i^h \in G_i, T_j^l \in
G_j\}$, where $\comp {T_i^h}$ denotes $\chi_i \setminus T_i^h$ . 
Now, we
claim that such a random $\F$ handles $A \sim \binom {[n]} {n/2}$ with
probability at least $1/2$, giving the existence of the desired
family. As there are $\Theta(d^2)$ pairs $(i, j)$ to consider and for each one we
add $\tilde\Theta((n^{1/2} d^{\semifinaldpow/2})^2)$ sets $S$ to $\F$,
this gives a total size $|\F| = \tilde\Theta(n d^\finaldpow)$.

  For  $I \subseteq [0,d]$ 
  we introduce an \defin{error term $R$(I)} 
  to represent the error in balancing $A$ . We let
$\chi(I) = \cup_{i \in I} \chi_i$ and
$R(I) = |A \cap \chi(I)| - \frac{|\chi(I)|} 2$. Furthermore we write $R_i := R(\{i\})$. For reasons that will
become clear later, we want to choose a permutation $\pi$ of $[0,d]$
with $\pi(0) = 0$ and $\pi(d) = d$ with
$\max_{i \in [0, d]} |R(\pi([0, i]))|$ small.

\begin{figure}[h!]
  \label{fig:exGalvin}
    \centering
    \includegraphics[scale=2]{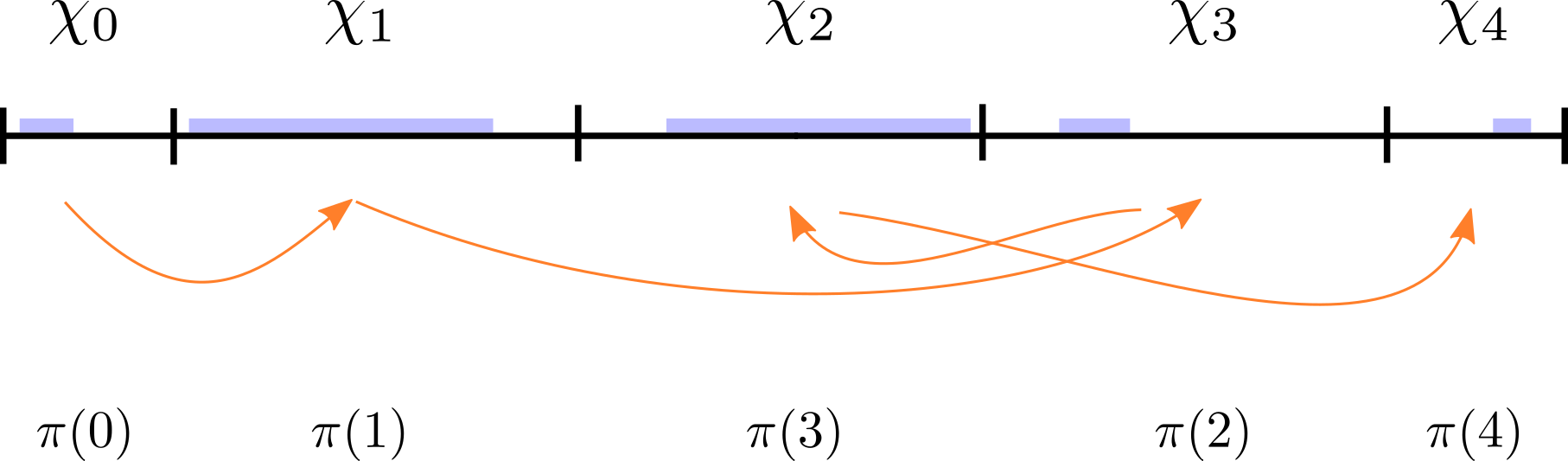}
    \caption{An ordering $\pi$}
  \end{figure}

\begin{claim}\label{goodpi} $\exists \pi : \max_{i \in [0, d]} |R(\pi([0, i]))| \le
\max_{i \in [0,d]} |R_i|$
\end{claim}
\begin{proof}
  We let $\pi(0)$ be fixed to be $0$, and for each $i \geq 0$, pick $\pi(i+1)$
  among the remaining elements such that $R_{\pi(i+1)}$ has opposite
  sign from $R(\pi[0,i])$. If $R(\pi[0, i]) = 0$ pick any value of $\pi(i+1)$. Note that this is always possible as
  $R([0,d]) = 0$.
\end{proof}

We fix $\pi$ to be a permutation that fulfills Claim~\ref{goodpi} for the rest of the paper.

\begin{claim}\label{low_max_lem} With probability at least $\frac 3 4$ we have $\max_{i \in [0,d]} |R_i|
  \le \sqrt{\ln (\maxlemconst
   d)} \sqrt k$.
\end{claim}
\begin{proof}

  For $i \in [1,d-1]$, each element $R_i$ follows a hypergeometric distribution $H(\frac n 2, n, 2k)$. We get the following bound, due to Hoeffding~\cite{H63}:

    $$P(|R_i| > x) \le 2 \exp(-\frac {2x^2} {2k}) $$

    With $x = \sqrt{ \ln (\maxlemconst d) } \sqrt k$ this becomes
    $2\exp(-\ln (\maxlemconst d)) = \frac 2 {\maxlemconst} \cdot \frac 1 d$. $R_0$
    and $R_d$ follow the distribution $H(\frac n 2, n, k)$,
    which yields an even stronger bound for $i = 0$ and $i =
    d$. Applying a union bound over all $i \in [d]$, the probability
    that at least one $|R_i|$ exceeds
    $\sqrt{\ln (\maxlemconst d)} \sqrt k$ is bounded by
    $\frac 2 \maxlemconst \frac {d+1} d < \frac 1 4$ (since $d \geq
    2$).
\end{proof}

\begin{claim}\label{sample_lem} Suppose $d < \frac n {(\ln n)^3}$. Given some $T_i \in G_i$ for $i \in
[1, d]$, let $S_j := \comp T_{\pi(j-1)} \cup
T_{\pi(j)}$ for $j \in [d]$. If $\{S_j\}_{j < i}$ are balanced on $A$
  then we have $S_i$ balanced on $A$ with probability at least
  $$\Theta\left( \exp(-\frac \dpowparam k \max \{R(\pi[0,i-1])^2, R_{\pi(i)}^2\})
  \sqrt{\frac 1 k} \right)$$
\end{claim}
\begin{proof}

  Let $t := -R(\pi[0,i-1])$. Since the $\{ S_j\}_{j < i}$ are balanced, we have:
  \begin{equation}
    \label{eq:bal}
    |A \cap \cup_{j=1}^{i-1} S_j| = (i-1) k
\end{equation}

On the other hand:
\begin{align*}
  |A \cap \chi(\pi[0,i-1])| &= |A \cap \cup_{j=1}^{i-1} S_j| + |A \cap
                              \comp T_{\pi(i-1)}| & \\
                            &= (i-1) k + |A \cap
                              \comp T_{\pi(i-1)}| & \text{using
                                                    (\ref{eq:bal})} \\
  \intertext{and}
  |A \cap \chi(\pi[0,i-1])| &= (2i-1)\frac k 2 - t &
                                                                    \text{by
                                                                    definition
                                                                    of
                                                                    }
                                                                    R(\cdot)
\end{align*}
Therefore,
  $|A \cap \comp T_{\pi(i-1)}| = \frac k 2 - t$. To make $S_i$ to be balanced we
must have
$|A \cap T_{\pi(i)}| + |A \cap \comp T_{\pi(i-1)}| = k$. This means that the probability that $S_i$ is balanced is the
  probability that $|A \cap T_{\pi(i)}| = \frac k 2 + t$. 
  Let $x := |A \cap T_{\pi(i)}|$ and $R := R_{\pi(i)}$. We
    have that $x$ follows a hypergeometric distribution with
    parameters $H(k + R, 2k, k)$. Claim~\ref{algebra} below suffices to establish Claim~\ref{sample_lem}.
\end{proof}

\begin{figure}[h!]
  \label{fig:exGalvin}
    \centering
    \includegraphics[scale=2.8]{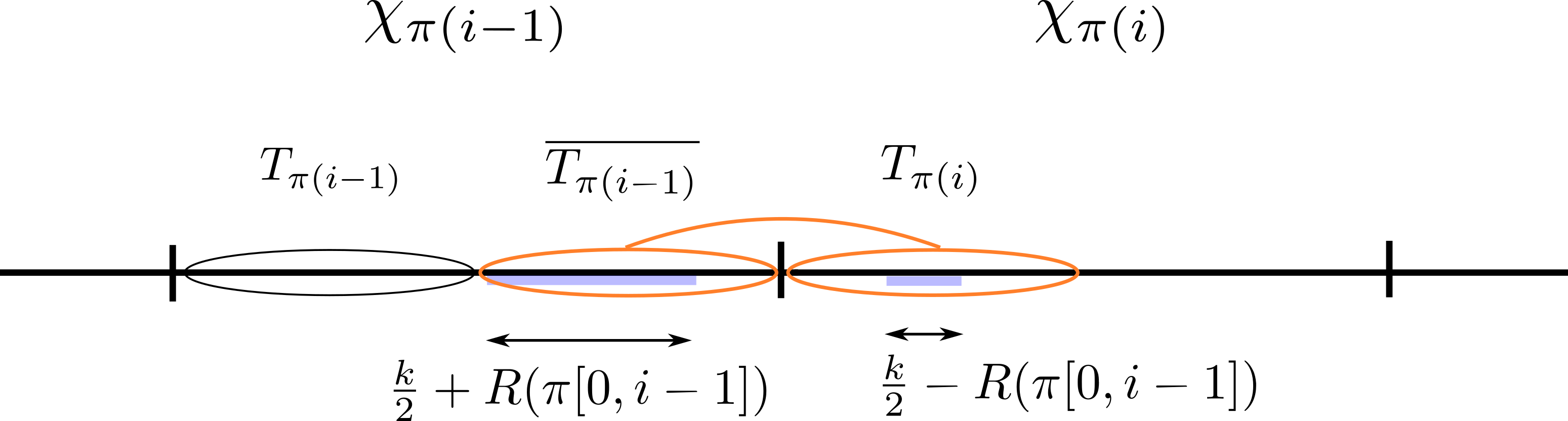}
    \caption{Conditions on $|A \cap T_{\pi(i)}|$}
  \end{figure}

  We state an easy lemma that will be helpful for Claim~\ref{algebra}
  to estimate binomial coefficients, a proof of which can be found in
  Spencer and Florescu~\cite{SF14}.
\begin{lemma}
  \label{lem:binomial}
  $\binom {n} {\frac n 2 - m} = 2^n \sqrt{\frac{2}{n\pi}}
  \exp\left(-\frac{2m^2}{n}\right)\left(1 + O(\frac{m^3}{n^2})\right)$
\end{lemma}

  \begin{claim}
    \label{algebra}
    We have that $x = \frac k 2 + t$ with probability at least
    $$\Theta \left( \exp(-\frac \dpowparam k \max \{t^2, \frac {R^2} 4\})
    \sqrt{\frac 1 k} \right)$$
  \end{claim}

\begin{proof}
  As $x$ follows a hypergeometric distribution with parameters
  $H(k+R,2k,k)$, we have that
  \begin{equation}
  \label{eq:px}
  P(x = \frac k 2 + t) = \binom {k+R} {\frac k 2 + t}
  \binom {k-R} {\frac k 2 - t} {\binom {2k} k}^{-1}.
  \end{equation}
    
  As long as $(\frac R 2 - t)^3 = o(k^2)$, which is the case when $d <
  \frac{n}{(\ln n)^3}$, we may apply
  Lemma~\ref{lem:binomial}, we have that (\ref{eq:px}) equals

\begin{align*}
    =& 2^{k+R} \sqrt{\frac{2}{(k+R)\pi}}
  \exp\left(-\frac{2(\frac R 2 - t)^2}{k+R}\right) \\
    &\times 2^{k-R} \sqrt{\frac{2}{(k-R)\pi}}
  \exp\left(-\frac{2(\frac R 2 - t)^2}{k-R}\right) \\
    &\times \left( 2^{2k} \sqrt{\frac{2}{2k\pi}} \right)^{-1}
    (1 + o(1)) \\
    =& \sqrt{\frac {4k} {(k+R)(k-R) \pi}}
    \exp \left( -2 (\frac R 2 - t)^2 (\frac 1 {k+R} + \frac 1 {k-R}) \right)
    (1 + o(1)) \\
    =& \sqrt{\frac {4k} {(k^2 - R^2)\pi}}
    \exp \left( \frac {-4k(\frac R 2 - t)^2} {k^2 - R^2} \right)
    (1 + o(1)) \\
    \intertext{By Claim~\ref{low_max_lem} we have $0 \le t, R \le
  \sqrt{ \ln(\maxlemconst d)} \sqrt k = o(k)$, therefore we finally get}
    =& \sqrt{\frac 4 {k\pi}} \exp\left(- \frac 4 k (\frac R 2 - t)^2\right) (1 + o(1))
  \end{align*}
\end{proof}

Combining Claim~\ref{low_max_lem} and Claim~\ref{sample_lem}, we have a probability of
\begin{align*}
  \Theta \left( \exp(-\frac \dpowparam k (k \ln (\maxlemconst d)))
    \sqrt{\frac 1 k} \right)
  &= \Theta \left( \exp(- \dpowparam \ln (\maxlemconst d)) \sqrt{\frac d n} \right)\\
  &= \Theta((\maxlemconst d)^{-\semifinaldpow/2} n^{-1/2})
\end{align*}
that $S_i$ is balanced. Call this probability $y$. If
$|G_i| = \frac {\ln(4d)} y$ then the probability that some choice of
$T_{\pi(i)}$ balances $S_i$ is at least $1 - \frac 1{4d}$. By the
union bound, the chance that $|R_i|$ is not bounded in
Claim~\ref{low_max_lem} or that any $S_{i}$ is unbalanced is at most
$\frac 1 4 + d \frac 1{4d} = \frac 1 2$. Hence the probability that we
get a $d$-Galvin partition is at least $\frac 1 2$, as desired.

\end{proof}

  In the above proof we used $d < \frac n {(\ln n)^3}$ to apply Lemma~\ref{lem:binomial}. While this could perhaps be improved to $d = \frac n {\ln n}$, there is a real barrier here. When $d$ is this large we expect some buckets to be entirely empty of elements from $A$ and the above proof does not work. We now handle the case where $d$ is larger.

\noindent{\textbf{Second case: }$d \geq \frac n {(\ln n)^3}$}
 
\begin{proof}[Proof of Theorem~\ref{thm:main}, second case]
  First, observe that Galvin families compose nicely; if $\F$ is an
  $a$-Galvin family over $[n]$, and if we take a $b$-Galvin family
  $\F_S$ over $S$ for each set $S \in \F$, then the union of all
  $\F_S$ forms an $ab$-Galvin family.

  Set $d' = \frac{n}{(\ln n)^3} $ and assume for the moment that $d'$
  and $\frac d {d'}$ are valid factors of $d$. The idea is to start by
  constructing a $d'$-Galvin family $\F$ over $[n]$, using the
  previous construction. We then recursively apply the construction to
  get a $\frac d {d'}$-Galvin family $\F_S$ for any each $S \in \F$,
  and the final family is the union of all $\F_S$. The elements of
  $\F$ are sets of size $(\ln n)^3$, therefore the families $\F_S$ are
  of size $\tilde \Theta(1)$, and the overall construction is of size
  $\tilde \Theta(n^2 d^\finaldpow)$.

  In the case that $d'$ and $\frac d {d'}$ are not valid factors of
  $d$, we do the following. Let $k' = \lfloor \frac d {d'}
  \rfloor$. The idea is to construct a family $\F$ with sets of size
  $2k'k$, and $2(k'+1)k$, that behaves like a Galvin family: we ask
  that any set $A$ has a partition of $[n]$ from sets in $\F$, where
  each set of the partition is balanced on $A$. We then apply
  recursively the construction to split the sets of size $2k'k$ and
  $2(k'+1)k$ until we get size $k$ sets. To create the family $\F$, we
  adapt the construction of the Galvin family when
  $d < \frac{n}{(\ln n)^3}$, in the following way. Note that in any
  partition of $[n]$ into sets of these sizes, the number of sets of
  size $2k'k$ and $2(k'+1)k$ are fixed (given by $d$ and $n$). We
  denote these numbers by $f$ and $c$. We need to ensure that the
  $\comp {T_i^h} \cup T_j^l$ are of the correct sizes (\ie $2k'k$ or $2(k'+1)k$). For that, we change the sizes of the
  $\chi_i$ in the following way:
  \begin{itemize}
\item $|\chi_0| = k'k$
  \item For $c$ values of $i \in [1,d-1]$, we have $|\chi_i| = 2(k'+1)k$
\item For the other $i \in [1,d-1]$ we have $|\chi_i| = 2k'k$
\item $|\chi_d| = k'k$.
  \end{itemize}
  We then choose the $T_i^h$ to be of size $k'k$ except for $i=0$
  where the unique $T_0$ remains $\emptyset$. This gives the desired sizes for $|S_i|$ and
  it is not hard to see that the proof carries over to this case
  with some simple and obvious modifications.

\end{proof}

\subsection{Galvin family without the divisibility condition}
The previous definition of a $d$-Galvin family requires $2d \mid
n$. Here we present a relaxed version, which can be defined without
the divisibility condition, and prove that such families of polynomial
size can be obtained using our previous construction.

When the divisibility condition does not hold we would like $d$ sets to be exactly or almost exactly balanced on $A$ and for those sets to be as close in size as possible. To be exactly balanced they must have evenly many elements, so if $[n]$ is odd then we must include a set of odd size which is imbalanced by 1 element. Of the remaining elements, the closest they can come in size is differing by 2 elements - being of size either $2\lfloor k \rfloor$ or $2\lceil k \rceil$. We are able to achieve this best possible outcome.

\begin{definition}[$d$-Galvin family, second version]
  Given two integers $d, n \in \N$ with $d \leq n$, we say that a family
  $\F \subseteq 2^{[n]}$ is \defin{$d$-Galvin} if for any
  $A \in \binom {[n]} {\lceil n/2 \rceil}$, \defin{$A$ is
    handled by $\F$}, meaning that there exist $d$ sets
  $S_1,\dots,S_d \in \F$ such that:
  \begin{enumerate}
  \item $\forall i < d$, $|S_i| = 2 \lfloor k \rfloor$ or
    $|S_i| = 2 \lceil k \rceil$,
    \item $2 \lfloor k \rfloor \leq |S_d| \leq 2 \lceil k \rceil$
  \item The $S_i$ form a partition of $[n]$,
  \item For $i < d$, each $S_i$ is balanced on $A$.
  \item $|\comp A \cap S_d| \leq |A \cap S_d| \leq |\comp A \cap S_d| + 1$.
  \end{enumerate}
\end{definition}

\begin{figure}[h!]
  \label{fig:exGalvin}
    \centering
    \includegraphics[scale=1.3]{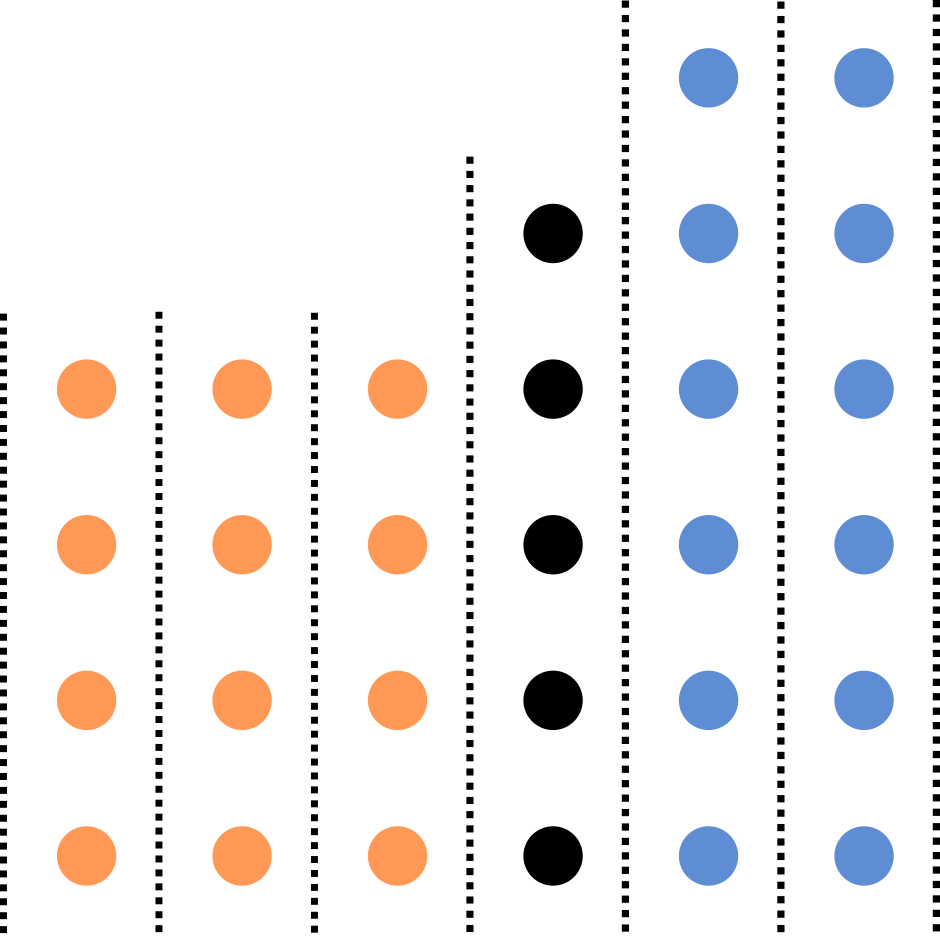}
    \caption{For $n = 29, d = 6$, we have three sets of size
      $2 \lfloor k \rfloor$, two sets of size $2 \lceil k \rceil$, and
      one set of size $\lfloor k \rfloor + \lceil
      k \rceil$.}
  \end{figure}

\begin{theorem}
  There exists a $d$-Galvin family of size polynomial in $d$ and $n$.
\end{theorem}

\begin{proof}[Sketch of the proof.]
  We modify the previous construction slightly in order to handle this
  more general setting. This is very similar to the proof of
  Theorem~\ref{thm:main} in the case $d \geq \frac{n}{(\ln
    n)^3}$. Suppose $k$ is not an integer and write
  $k' := \lfloor k \rfloor$. Furthermore, assume for the moment that
  $k = \omega((\ln n)^3)$ so that the construction from
  Claim~\ref{sample_lem} holds. Note that in any partition of $[n]$
  into sets that respect properties $(1)$ and $(2)$ of the definition,
  the number of sets of size $2k'$, $2k' +1$, and $2(k'+1)$ are fixed
  (given by $d$ and $n$). We denote these numbers by $f, m$ and
  $c$. We need to ensure that the $\comp {T_i^h} \cup T_j^l$ are of
  the correct size in order to be able to fulfill our definition. For
  that, we change the size of the $\chi_i$ in the following way:
  \begin{itemize}
\item $|\chi_0| = k'$ if $m = 0$ and $k'+1$ otherwise
  \item For $c$ values of $i \in [1,d-1]$, we have $|\chi_i| = 2(k'+1)$
\item For the other $i \in [1,d-1]$ we have $|\chi_i| = 2k'$
\item $|\chi_d| = k'$.
  \end{itemize}
  We then choose the $T_i^h$ to be of size $k'$ except for $i=0$ where
  the unique $T_0$ remains $\emptyset$. By doing so, the partitions
  from the family respect properties $(1)$ and $(2)$, and again 
  the proof that this gives a valid construction is very close
  to the original proof and we omit the details.

  Finally, if $k = O((\ln n)^3)$ then we may have to simultaneously
  apply the adjustments above and the ones in the proof of the second
  case of Theorem~\ref{thm:main}.

  \end{proof}

\section{Discussion and open questions}
The actual construction is probabilistic and it could be interesting to
derandomize it, without increasing too much the size of the family. A
way to tackle the problem is to carefully design the sets $T_i$
belonging to $G_i$ instead of taking them randomly.

The given upper bound is nicely polynomial in $n$ and
$d$ but it is unlikely to be tight.
We suspect that even modifications of the current
construction can yield some improvements.  In particular,
the family $\F$ from Lemma~\ref{lem:randomf} is constructed by
taking the union $\comp T_i \cup T_j$ over all possible pairs
$(T_i,T_j) \in G_i \times G_j$ for $i,j \in [d]$. It might be
possible to 
restrict $(i,j)$ to come from the edges of a sparse graph over the
vertices $[d]$, and still prove
Claim \ref{goodpi}, maybe in
some slightly weaker form, possibly saving a factor close to $d$.
Even if this is possible the resulting family is still not likely
to be optimal size and hence we have not investigated this approach
in detail as it would lead to considerable complications and
we prefer a simple construction.  A truly optimal construction 
is likely to require some new ideas.

While
there is a linear lower bound for the original Galvin problem, it is
not clear how to derive from this linear lower bounds for $d$-Galvin
families for $p > 2$. An easy counting argument, similar to the one
for the original Galvin problem, gives that
$|\F|^{d-1} \geq \frac{\binom {n} {n/2}}{\binom {n/d} {k}^d}$
(since the number of possible partitions of $[n]$ with
$d$ sets from $\F$ is bounded by $|\F|^{d-1}$), providing $|\F| \geq
\Omega(\frac {\sqrt n} {d^{\frac 1 2 -\frac 1 {2d}}})$.
When focusing on large $d$ we get the simple bound below which is
an improvement in the regime
$d = \Omega(n^{1/5})$:

  \begin{claim} A $d$-Galvin family must be size at least $\frac {d^2} 2$.
  \end{claim}
  \begin{proof}
    Let us fix a $d$-Galvin family $\F$ over $[n]$, and consider the
    set $B = \{(S,x), S \in \F, x \in S\}$.

    We first prove that for any $x \in [n]$, there must be
      at least $\frac d 2$ sets from $\F$ that contain $x$. Suppose it
      is not the case for a particular $a \in [n]$, and consider a set
      $A$ of size $\frac n 2$ that contains
      $(\cup_{S \text{ s.t } a \in S} S)$ (such a $A$ exists since by
      the assumption the union is smaller than or equal to
      $\frac n 2$). Any set $S\in \F$ that contains $a$ is completely
      included in $A$, and thus cannot be balanced on $A$. Therefore
      $A$ is not handled by $\F$.

      Finally, observe that the previous remark implies that
      $|B| \geq \frac{nd}{2}$. As each set $S \in \F$ is of size
      $\frac n d$, the number of sets in $\F$ must be at least
      $\frac{d^2}{2}$.
  
  \end{proof}

\section*{Acknowledgements}
We thank Andrew Morgan for giving helpful suggestions in the details
of claims 4 and 5. The second author would like to thank Hervé
Fournier for valuable discussions.
\bibliographystyle{plain}

\bibliography{references}

\end{document}